\documentclass[12pt]{amsart}

\usepackage[utf8]{inputenc}%(only for the pdftex engine)

\usepackage{amsmath,amssymb,amsthm,amsfonts,graphicx}
\usepackage{subcaption}

\newtheorem{thm}{Theorem}[section]
\newtheorem{cor}[thm]{Corollary}
\newtheorem{lem}[thm]{Lemma}
\newtheorem{prop}[thm]{Proposition}
\newtheorem{rem}[thm]{Remark}

\newcommand{\simp}{\triangle}

\begin{document}
 
  %\articletype{Research Article{\hfill}Open Access}
  \author{Steven N. Harding}
  \address{Iowa State University, Department of Mathematics, 411 Morrill Road, Ames, IA 50011, U.S.A.; E-mail: sharding@iastate.edu}
  \author{Alexander W. N. Riasanovsky}
  \address{Iowa State University, Department of Mathematics, 411 Morrill Road, Ames, IA 50011, U.S.A.; E-mail: awnr@iastate.edu}
  \title{Moments of the weighted Cantor measures}
  %\runningtitle{Moments of the weighted Cantor measures}
  %\subtitle{...}
  \begin{abstract}
{
Based on the seminal work of Hutchinson, we investigate properties of {\em $\alpha$-weighted Cantor measures} whose support is a fractal contained in the unit interval.  Here, $\alpha$ is a vector of nonnegative weights summing to $1$, and the corresponding weighted Cantor measure $\mu^\alpha$ is the unique Borel probability measure on $[0,1]$ satisfying $ \mu^\alpha(E) = \sum_{ n=0 }^{N-1} \alpha_n\mu^\alpha( \varphi_n^{-1}(E) )$ where $\varphi_n: x\mapsto (x+n)/N$.  In Sections \ref{sec intro} and \ref{sec properties} we examine several general properties of the measure $\mu^\alpha$ and the associated Legendre polynomials in $L_{\mu^\alpha}^2[0,1]$.  In Section \ref{sec moments}, we (1) compute the Laplacian and moment generating function of $\mu^\alpha$, (2) characterize precisely when the moments $I_m = \int_{[0,1]}x^m\,d\mu^\alpha$ exhibit either polynomial or exponential decay, and (3) describe an algorithm which estimates the first $m$ moments within uniform error $\varepsilon$ in $O( (\log\log(1/\varepsilon))\cdot m\log m )$.  We also state analogous results in the natural case where $\alpha$ is {\em palindromic} for the measure $\nu^{\alpha}$ attained by shifting $\mu^{\alpha}$ to $[-1/2,1/2]$.  
}
\end{abstract}
 % \communicated{...}
 % \dedication{...}
  
\maketitle
\section{Introduction}\label{sec intro}

In the seminal paper \cite{Hutch}, Hutchinson realized a fractal as the invariant compact set, called the \textit{attractor}, of an iterated function system (IFS), i.e. a family of contraction maps on a complete metric space. Specifically, given an IFS $\{\varphi_n\}_{n=0}^{N-1}$ on $X$, the attractor of the IFS is the unique compact set $K \subset X$ satisfying
\[
K = \bigcup_{n=0}^{N-1}\varphi_n(K).
\]
Hutchinson showed the existence and uniqueness of a self-similar Borel probability measure supported on the attractor of an IFS. We denote by $\simp_N$ the standard simplex in $\mathbb{R}^N$ and  $\simp_N^*\subseteq \simp_N$ consisting of $\alpha = (\alpha_0,\alpha_1,...,\alpha_{N-1}) \in \simp_N$ such that $\alpha_n < 1$ for all $n$ and call elements of $\simp_N$ {\em weight vectors}. We now paraphrase Hutchinson's result.

\begin{thm}[Hutchinson, \cite{Hutch}]
Suppose $\{\varphi_n\}_{n=0}^{N-1}$ is an IFS on a complete metric space $X$ with attractor $K$, and let $\alpha\in \simp_N$.  There exists a unique Borel regular measure $\mu^{\alpha}$ on $X$ supported on $K$ such that
\begin{equation}\label{invariance}
\mu^{\alpha}(E) = \sum_{n=0}^{N-1}\alpha_n\mu^{\alpha}\left(\varphi_n^{-1}(E)\right)
\end{equation}
for all Borel-measurable $E\subseteq X$.
\end{thm}

Using the terminology of \cite{JorgKornShu}, we refer to the measure $\mu^{\alpha}$ as the {\em $\alpha$-equilibrium measure} when $X = \mathbb{R}^n$ or $X = \mathbb{C}$. We will call the measure $\mu^{\alpha}$ an {\em $\alpha$-weighted Cantor measure} when the associated IFS $\{\varphi_n\}_{n=0}^{N-1}$ on $\mathbb{R}$ is given by $\varphi_n : x \mapsto (x + n)/N$. An equilibrium measure is described as having {\em maximal entropy} if the associated weights are uniform, i.e. $\alpha_n$ is either $0$ or $1/k$ for each $n$.  An equilibrium measure that has attracted a lot of interest in the non-smooth harmonic analysis community is the ternary Cantor measure which arises from the weight vector $\alpha = (1/2,0,1/2)$. In \cite{JorgPed}, Jorgensen and Pedersen addressed the question of when a maximal entropy equilibrium measure $\mu^\alpha$ is {\em spectral}, that is, if there exists some countable set $\Lambda \subset \mathbb{R}$ so that the complex exponential functions $\{e^{2\pi i\lambda x}\}_{\lambda \in \Lambda}$ form an orthonormal basis for the Hilbert space $L^2_{\mu^\alpha}[0,1]$.  Jorgensen and Pedersen found that, while the quaternary Cantor measure corresponding to $\alpha = (1/2,0,1/2,0)$ is spectral, the ternary Cantor measure is not.
\\
\\
Much effort has been made to remedy this artifact of the ternary Cantor measure.  In \cite{DPS}, Dutkay, Picioroaga, and Song constructed an orthonormal basis consisting of piecewise exponentials on the ternary Cantor set.  Strichartz in \cite{Strich} posed the question of the existence of a frame, which is a generalization of an orthonormal basis, on the ternary Cantor set; however, this problem remains open.  Polynomial function systems provide a tempting alternative.  To this end, we define the {\em Legendre polynomials} in $L^2_{\mu^\alpha}[0,1]$ to be the result of applying the Gram-Schmidt algorithm to any sequence of polynomials of degrees $0,1,2,\dots$, respectively.  At each step, it becomes necessary to compute inner products of the form $\int_{[0,1]} x^m\,d\mu^\alpha(x)$.  These quantities, better known as the \textit{moments} of the measure $\mu^\alpha$, have elicited a lot of attention. Dovgoshey, Martio, Ryazanov, and Vuorinen provide a fairly comprehensive survey of the ternary Cantor function, including moments of the measure for which it is the distribution, in \cite{Dovgoshey}; Jorgensen, Kornelson and Shuman in \cite{JorgKornShu} study the moments of  equilibrium measures through an operator theory perspective using infinite matrices.
\\
\\
Our main results are as follows.  In Section \ref{sec properties}, we make the connection of these measures to a result by Pei, showing that the weighted Cantor measures are singular except in the trivial case of $\alpha_n = 1/N$ for all $n$ when the measure is Lebesgue.  We then provide more content in the way of characterizing these measures.  In Proposition \ref{poly}, we prove a generalization of Bonnet's recursion formula for orthogonal polynomial systems.  In Theorem \ref{entire}, we derive an explicit infinite product formula for the Laplacian (and thus the moment generating function) of $\mu^\alpha$ and estimate in Theorem \ref{cauchy est} the rapid convergence of the coefficients of the partial product.  This leads to Remark \ref{algorithm} which outlines a $O(\log\log(1/\varepsilon)\cdot m\log m)$ algorithm for estimating the first $m$ moments to uniform error at most $\varepsilon>0$.

\section{Properties of the weighted Cantor measure}\label{sec properties}
Our first observation motivates the distinction of $\simp_N^*$ from the simplex $\simp_N$. It is a direct consequence of the uniqueness of a Borel measure satisfying the invariance relation in Equation (\ref{invariance}), and the proof is omitted.

\begin{prop}
Suppose $\alpha \in \simp_N$ with $\alpha_n = 1$ for some $n$. Then $\mu^{\alpha}$ is the Dirac measure centered at $n/(N - 1)$, the fixed point of $\varphi_n^{-1}$.
\end{prop}

Given a finite Borel measure $\mu$ on $\mathbb{R}$, the cumulative distribution function (CDF) $F_\mu(x):=\mu(-\infty,x]$ is the increasing, right-continuous function which uniquely determines the measure. Therefore, to understand the weighted Cantor measure $\mu^{\alpha}$, it is useful to note some basic properties of $F_{\mu^\alpha}$.

\begin{prop}\label{CDF}
Fix $\alpha \in \simp_N^*$, and let $k$ be a positive integer. For $n_\ell \in \{0,1,...,N-1\}$,
\begin{align}\label{CDFident}
F_{\mu^{\alpha}}\left(\dfrac{1}{N^k}\left[1+\sum_{\ell=0}^{k-1}n_\ell N^\ell\right]\right) - F_{\mu^{\alpha}}\left(\dfrac{1}{N^k}\sum_{\ell=0}^{k-1}n_\ell N^\ell\right) = \prod_{\ell=0}^{k-1}\alpha_{n_\ell}.  
\end{align}
\end{prop}

\begin{proof}
From the invariance relation in Equation ($\ref{invariance}$), we note that the CDF satisfies
\begin{align}\label{CDFinvar}
F_{\mu^{\alpha}}(x) = \sum_{n=0}^{N-1}\alpha_nF_{\mu^{\alpha}}(Nx-n).
\end{align}
Then, since $F_{\mu^{\alpha}}$ is the CDF of a measure supported in $[0,1]$, we have $F_{\mu^{\alpha}}(0) = \alpha_0F_{\mu^{\alpha}}(0)$ which implies that $F_{\mu^{\alpha}}(0) = 0$. Equation $(\ref{CDFident})$ for $k = 1$ immediately follows from this observation and Equation ($\ref{CDFinvar}$). We proceed by induction on $k$. Applying Equation (\ref{CDFinvar}), we have
\begin{align*}
F_{\mu^{\alpha}}&\left(\dfrac{1}{N^{k+1}}\left[1+\sum_{\ell=0}^{k}n_\ell N^\ell\right]\right) - F_{\mu^{\alpha}}\left(\dfrac{1}{N^{k+1}}\sum_{\ell=0}^{k}n_\ell N^\ell\right) \\
& = \sum_{n=0}^{N-1} \alpha_n\left\{F_{\mu^{\alpha}}\left(\dfrac{1}{N^k}\left[1+\sum_{\ell=0}^{k-1}n_\ell N^\ell\right] + n_k - n\right) - F_{\mu^{\alpha}}\left(\dfrac{1}{N^k}\left[\sum_{\ell=0}^{k-1}n_\ell N^\ell\right] + n_k - n\right)\right\} \\
& = \alpha_{n_k}\left\{F_{\mu^{\alpha}}\left(\dfrac{1}{N^k}\left[1+\sum_{\ell=0}^{k-1}n_\ell N^\ell\right]\right) - F_{\mu^{\alpha}}\left(\dfrac{1}{N^k}\sum_{\ell=0}^{k-1}n_\ell N^\ell\right)\right\} \\
& = \alpha_{n_k}\prod_{\ell=0}^{k-1}\alpha_{n_\ell}
\end{align*}
which concludes the induction.

\end{proof}

Proposition \ref{CDF} readily implies that the monotone functions constructed by Pei in \cite{Pei} are identical to the CDF's of the weighted Cantor measures.  %Additionally, Pei constructed a class of monotone increasing functions from the distributions of random variables. These distributions satisfy identity $(\ref{CDFident})$ and are therefore the CDF of a weighted Cantor measure. 
Pei therefore proved results pertaining to differentiability and H\"older continuity of $F_{\mu^\alpha}$. We paraphrase those results.

\begin{thm}[Pei, \cite{Pei}]
Let $\alpha \in \simp_N$. $F_{\mu^{\alpha}}$ is strictly increasing unless $\alpha_n = 0$ for some $n$ and is H\"older continuous with the exponent $\log(1/r)/\log(N)$ where $r = \max\{\alpha_0,\alpha_1,...,\alpha_{N-1}\}$. Furthermore, $F_{\mu^{\alpha}}$ is singular continuous except when $\alpha$ is the uniform distribution $(1/N,...,1/N)$ in which case $F_{\mu^{\alpha}}(x) = x$.
\end{thm}

%It is clear from Proposition \ref{CDF} that $\alpha \in \simp_N$ uniquely determines a weighted Cantor measure; however, the converse is not true. That is, a weighted Cantor measure arises from infinitely many different probability vectors (of various lengths).

Recall that the weighted Cantor measure is determined by weighting, scaling and translating under the IFS according to the invariance relation in Equation (\ref{invariance}). The next proposition illustrates that this invariant condition applies as well to the weight vector. Precisely, there are $\alpha\in \simp_M$ and $\beta\in\simp_N$ with $M \neq N$ so that $\mu^\alpha = \mu^\beta$.

\begin{prop}\label{equivmeas}
Fix $\alpha \in \simp_N$. Let $\beta = \alpha^{\otimes k}$, the Kronecker product of $\alpha$ with itself $k$ times.  Then $\mu^\alpha = \mu^\beta$.  
%let $\beta$ be the weight vector of length $N^k$ given by
%\[
%\beta_n = \prod_{j=0}^{k-1}\alpha_{n_j}
%\]
%where $\{n_j\}_{j=0}^{k-1}$ are the base-$N$ decomposition coefficients for $n$, i.e.
%\[
%n = \sum_{j=0}^{k-1}n_jN^j.
%\]
\end{prop}

\begin{proof}
It is readily checked that the element of $\beta$ indexed by $n = n_0 + n_1N + ... + n_{k-1}N^{k-1}$ where $n_\ell \in \{0,1,...,N-1\}$ is
\[
\beta_n = \prod_{\ell = 0}^{k-1} \alpha_{n_\ell}.
\]
The associated IFS for $\mu^{\beta}$ is $\{\psi_n\}_{n=0}^{N^k-1}$ where $\psi_n(x) = (x + n)/N^k$. Then, from the invariance relation in Equation (\ref{invariance}), we find
\[
\mu^{\beta}(E) = \sum_{n_0,n_1,...,n_{k-1}=0}^{N-1}\left(\prod_{\ell=0}^{k-1}\alpha_{n_\ell}\right)\mu^{\beta}(\psi_n^{-1}(E)).
\]
Since the IFS $\{\varphi_n\}_{n=0}^{N-1}$ for $\mu^{\alpha}$ is given by $\varphi_n(x) = (x + n)/N$, we have
\begin{align*}
\mu^{\alpha}(E) &= \sum_{n_{k-1} = 0}^{N-1}\alpha_{n_{k-1}}\mu^{\alpha}(\varphi_{n_{k-1}}^{-1}(E)) \\
&= \sum_{n_{k-1} = 0}^{N-1}\alpha_{n_{k-1}}\sum_{n_{k-2}=0}^{N-1}\alpha_{n_{k-2}}\mu^{\alpha}((\varphi_{n_{k-2}}^{-1}\circ\varphi_{n_{k-1}}^{-1})(E)) \\
&= ... = \sum_{n_0,n_1,...,n_{k-1}=0}^{N-1}\left(\prod_{\ell=0}^{k-1}\alpha_{n_\ell}\right)\mu^{\alpha}((\varphi_{n_0}^{-1}\circ ...\circ\varphi_{n_{k-2}}^{-1}\circ\varphi_{n_{k-1}}^{-1})(E)) \\
&= \sum_{n_0,n_1,...,n_{k-1}=0}^{N-1}\left(\prod_{\ell=0}^{k-1}\alpha_{n_\ell}\right)\mu^{\alpha}(\psi_n^{-1}(E)).
\end{align*}
By uniqueness of the measure, it follows that $\mu^{\alpha} = \mu^{\beta}$, as desired.

\end{proof}

For each positive integer $k$, we denote the sample $S_k \subset [0,1]$ as the set
\[
S_k := \left\{\dfrac{1}{N^k}\sum_{\ell=0}^{k-1}n_\ell N^\ell\,\middle|\,n_\ell \in \{0,1,...,N-1\}\right\}\cup\{1\}.
\]
Further, we define $F_{\mu^{\alpha},k} : [0,1] \rightarrow [0,1]$ to be the linear interpolation of the $N^k + 1$ many points $\{(x,F_{\mu^{\alpha}}(x))\,|\,x \in S_k\}$. Note, from Proposition \ref{equivmeas}, that $F_{\mu^{\alpha},k} = F_{\mu^{\beta},1}$ where $\beta = \alpha^{\otimes k}$.

\begin{prop}\label{unifconv}
Let $\alpha \in \simp_N^*$. The sequence $\{F_{\mu^{\alpha},k}\}$ converges uniformly to $F_{\mu^{\alpha}}$.  
\end{prop}

\begin{proof}
Let $r = \max\{\alpha_0,\alpha_1,...,\alpha_{N-1}\} < 1$. Let $\varepsilon > 0$, and choose an integer $k$ such that $r^k < \varepsilon$. We show that $\|F_{\mu^{\alpha},j} - F_{\mu^{\alpha},k}\|_{\infty} < \varepsilon$ for every integer $j \geq k$. Since $|F_{\mu^{\alpha},j} - F_{\mu^{\alpha},k}|(x)$ is continuous on $[0,1]$, there exists an $x \in [0,1]$ such that
\[
\|F_{\mu^{\alpha},j} - F_{\mu^{\alpha},k}\|_{\infty} = \left|F_{\mu^{\alpha},j}(x) - F_{\mu^{\alpha},k}(x)\right|.
\]
There are $n_\ell \in \{0,1,...,N-1\}$ such that
\[
\dfrac{1}{N^k}\sum_{\ell=0}^{k-1}n_\ell N^\ell \leq x \leq \dfrac{1}{N^k}\left(1 + \sum_{\ell=0}^{k-1}n_\ell N^\ell\right).
\]
Since $F_{\mu^{\alpha},k}$ and $F_{\mu^{\alpha},j}$ are linear interpolations of points belonging to $F_{\mu^{\alpha}}$, we have
\[
\left|F_{\mu^{\alpha},j}(x) - F_{\mu^{\alpha},k}(x)\right| \leq F_{\mu^{\alpha}}\left(\dfrac{1}{N^k}\left[1 + \sum_{\ell=0}^{k-1}n_\ell N^\ell\right]\right) - F_{\mu^{\alpha}}\left(\dfrac{1}{N^k}\sum_{\ell=0}^{k-1}n_\ell N^\ell\right) = \prod_{\ell=0}^{k-1}\alpha_{n_\ell} \leq r^k < \varepsilon.
\]
Therefore the sequence $\{F_{\mu^{\alpha},k}\}$ is uniformly Cauchy and, thus, converges uniformly to some continuous function $f$. Since $\{F_{\mu^{\alpha},k}\}$ converges pointwise to $F_{\mu^{\alpha}}$ on a dense set, we have $F_{\mu^{\alpha}} = f$ on a dense set. Then, because $F_{\mu^{\alpha}}$ is right-continuous and $f$ is continuous, we have $f = F_{\mu^{\alpha}}$.

\end{proof}

For illustration, we attain the graph of $F_{\mu^{\alpha},k}$ through $F_{\mu^{\beta},1}$ where $\beta = \alpha^{\otimes k}$, as stated above. The benefit of the latter is that it is somewhat simple to take the Kronecker product of vectors up to sufficient resolution in programs such as Mathematica, which was used to produce Figure \ref{CDFs}.

\begin{figure}[h]
\begin{subfigure}[b]{0.48\textwidth}
\includegraphics[width=\textwidth]{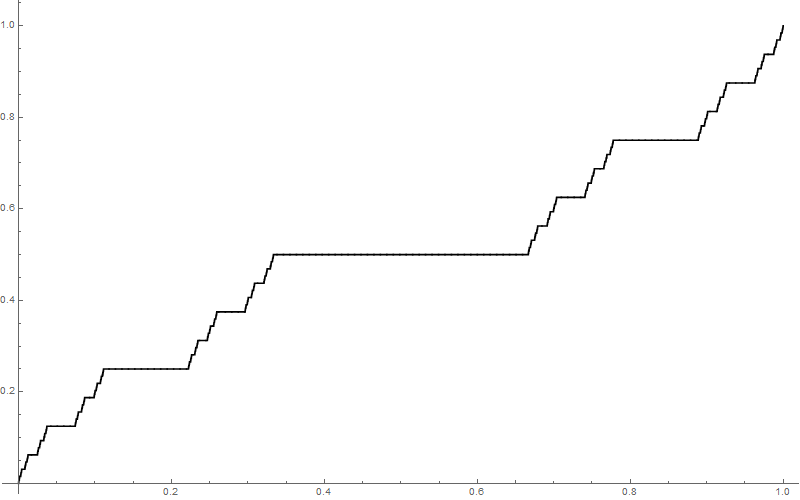}
\caption{$\alpha = (1/2,0,1/2)$}
\end{subfigure}
\begin{subfigure}[b]{0.48\textwidth}
\includegraphics[width=\textwidth]{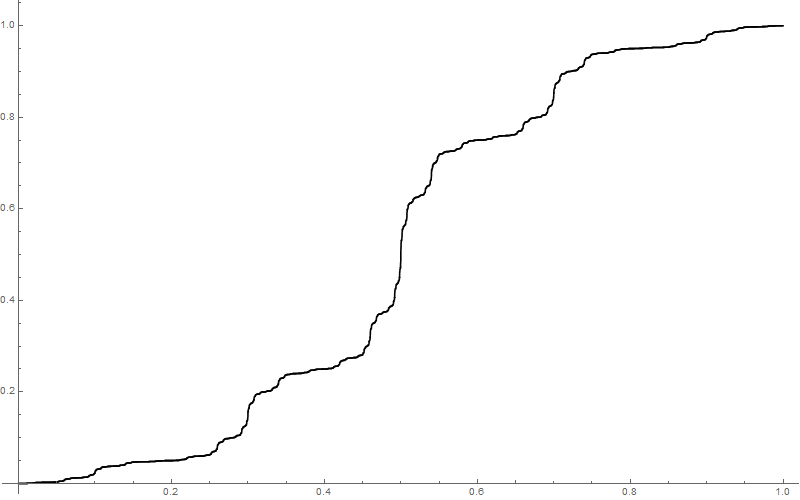}
\caption{$\alpha = (1/20,1/5,1/2,1/5,1/20)$}
\end{subfigure}
\caption{Graph of $F_{\mu^{\alpha}}$ for selected $\alpha$}
\label{CDFs}
\end{figure}

\newpage

\noindent The next results show that a small variation in $\alpha \in \simp_N^*$ leads to a relatively small variation in the corresponding measure. We start with a lemma which is pertinent to those results.

\begin{lem}\label{CDFapprox}
Fix positive integers $N$ and $k$. There exists a constant $c(N,k) > 0$ such that
\[
\|F_{\mu^{\alpha},k} - F_{\mu^{\beta},k}\|_\infty \leq c(N,k)\|\alpha - \beta\|_\infty
\]
for every $\alpha,\beta \in \simp^*_N$.
\end{lem}

\begin{proof}
Let $\alpha,\beta \in \simp^*_N$. Because $F_{\mu^{\alpha},k}$ and $F_{\mu^{\beta},k}$ are linear interpolations of $F_{\mu^{\alpha}}$ and $F_{\mu^{\beta}}$, respectively, on the set $S_k$, there exists a positive number $x \in S_k$ such that
\[
\|F_{\mu^{\alpha},k} - F_{\mu^{\beta},k}\|_\infty = |F_{\mu^{\alpha},k}(x) - F_{\mu^{\beta},k}(x)|.
\]
Suppose $n_\ell \in \{0,1,...,N-1\}$ such that
\[
x - \dfrac{1}{N^k} = \dfrac{1}{N^k}\sum_{\ell = 0}^{k - 1} n_{\ell}N^{\ell}.
\]
Then, by Proposition \ref{CDF}, we have
\begin{align*}
\|F_{\mu^{\alpha},k} - F_{\mu^{\beta},k}\|_{\infty} &= |F_{\mu^{\alpha},k}(x) - F_{\mu^{\beta},k}(x)| \\
&\leq \left|F_{\mu^{\alpha},k}\left(x - \dfrac{1}{N^k}\right) - F_{\mu^{\beta},k}\left(x - \dfrac{1}{N^k}\right)\right| \\
&\hspace{1cm} + \left|F_{\mu^{\alpha},k}(x) - F_{\mu^{\alpha},k}\left(x - \dfrac{1}{N^k}\right) - \left[F_{\mu^{\beta},k}(x) - F_{\mu^{\beta},k}\left(x - \dfrac{1}{N^k}\right)\right]\right| \\
&= \left|F_{\mu^{\alpha},k}\left(x - \dfrac{1}{N^k}\right) - F_{\mu^{\beta},k}\left(x - \dfrac{1}{N^k}\right)\right| + \left|\prod_{\ell = 0}^{k-1}\alpha_{n_\ell} - \prod_{\ell = 0}^{k-1}\beta_{n_\ell}\right|
\end{align*}
Repeating this argument sufficiently many times, we attain
\[
\|F_{\mu^{\alpha},k} - F_{\mu^{\beta},k}\|_{\infty} \leq \sum_{\vec{m}} \left|\prod_{\ell = 0}^{k-1}\alpha_{m_\ell} - \prod_{\ell = 0}^{k-1}\beta_{m_\ell}\right| \leq \left\|\alpha^{\otimes k} - \beta^{\otimes k}\right\|_1 \leq N^k\left\|\alpha^{\otimes k} - \beta^{\otimes k}\right\|_\infty
\]
where the sum ranges over $\vec{m} = (m_0,m_1,...,m_{k-1}) \in \{0,1,...,N-1\}^k$ satisfying
\[
\dfrac{1}{N^k}\sum_{\ell = 0}^{k-1}m_\ell N^\ell \leq x - \dfrac{1}{N^k}.
\]
We conclude the proof by showing the upper bound
\[
\left\|\alpha^{\otimes k} - \beta^{\otimes k}\right\|_{\infty} \leq \|\alpha - \beta\|_\infty\sum_{\ell = 0}^{k - 1}\|\alpha\|_\infty^\ell\|\beta\|_\infty^{k - 1 - \ell}.
\]
The inequality is trivial for $k = 1$. We proceed by induction on $k$. There are $m_\ell \in \{0,1,...,N-1\}$ such that
\begin{align*}
\left\|\alpha^{\otimes (k+1)} - \beta^{\otimes (k+1)}\right\|_\infty &= \left|\prod_{\ell=0}^k\alpha_{m_\ell} - \prod_{\ell=0}^k\beta_{m_\ell}\right| \\
&\leq \left|\prod_{\ell=0}^k\alpha_{m_\ell} - \beta_{m_k}\prod_{\ell=0}^{k-1}\alpha_{m_\ell}\right| + \left|\beta_{m_k}\prod_{\ell=0}^{k-1}\alpha_{m_\ell} - \prod_{\ell=0}^k\beta_{m_\ell}\right| \\
&= |\alpha_{m_k} - \beta_{m_k}|\left(\prod_{\ell=0}^{k-1}\alpha_{m_\ell}\right) + \beta_{m_k}\left|\prod_{\ell = 0}^{k-1}\alpha_{m_\ell} - \prod_{\ell=0}^{k-1}\beta_{m_\ell}\right| \\
&\leq \|\alpha - \beta\|_\infty\|\alpha\|_\infty^k + \|\beta\|_\infty\left\|\alpha^{\otimes k} - \beta^{\otimes k}\right\|_\infty \\
&\leq \|\alpha - \beta\|_\infty\|\alpha\|_\infty^k + \|\beta\|_\infty\|\alpha - \beta\|_\infty\sum_{\ell = 0}^{k - 1}\|\alpha\|_\infty^\ell\|\beta\|_\infty^{k - 1 - \ell} \\
&= \|\alpha - \beta\|_\infty\sum_{\ell = 0}^k\|\alpha\|_\infty^{\ell}\|\beta\|_\infty^{k-\ell}.
\end{align*}
This concludes the induction. Now, since $\|\alpha\|_\infty < 1$ and $\|\beta\|_\infty < 1$, we may let $c(N,k) = kN^k$.

\end{proof}

\begin{rem}
Given $\alpha,\beta \in \simp_N^*$, we note that $|F_{\mu^{\alpha}} - F_{\mu^{\beta}}|(x)$ need not attain the value $\|F_{\mu^{\alpha}} - F_{\mu^{\beta}}\|_\infty$ on the set $S_k$ for any $k$, e.g. $\alpha = (0,1/2,1/2)$ and $\beta = (1/2,1/2,0)$ where $\mu^{\alpha}$ is supported in $[1/2,1]$ and $\mu^{\beta}$ is supported in $[0,1/2]$.
\end{rem}

\begin{prop}\label{transform}
The transform $\alpha \mapsto F_{\mu^{\alpha}} : \simp_N^* \rightarrow C[0,1]$ is continuous.
\end{prop}

\begin{proof}
Suppose $\alpha \in \simp_N^*$, and let $\varepsilon > 0$. Since $F_{\mu^{\alpha}}$ is uniformly continuous, there exists a positive integer $k$ such that
\[
F_{\mu^{\alpha},k}\left(x + \dfrac{1}{N^k}\right) - F_{\mu^{\alpha},k}(x) = F_{\mu^{\alpha}}\left(x + \dfrac{1}{N^k}\right) - F_{\mu^{\alpha}}(x) \leq \dfrac{\varepsilon}{2}
\]
for every $x \in S_k\setminus\{1\}$. By Lemma \ref{CDFapprox}, there exists a $\delta > 0$ such that $\|F_{\mu^{\alpha},k} - F_{\mu^{\beta},k}\|_{\infty} \leq \varepsilon/2$ whenever $\|\alpha - \beta\|_{\infty} < \delta$. In particular, we have $\left|F_{\mu^{\alpha}}(x) - F_{\mu^{\beta}}(x)\right| \leq \varepsilon/2$ for every $x \in S_k$. Then, for $y \in [x,x + 1/N^k]$ where $x \in S_k\setminus\{1\}$, we find
\[
F_{\mu^{\alpha}}(x) - \dfrac{\varepsilon}{2} \leq F_{\mu^{\beta}}(x) \leq F_{\mu^{\beta}}(y) \leq F_{\mu^{\beta}}\left(x + \dfrac{1}{N^k}\right) \leq F_{\mu^{\alpha}}\left(x + \dfrac{1}{N^k}\right) + \dfrac{\varepsilon}{2}.
\]
It immediately follows that $|F_{\mu^{\alpha}}(y) - F_{\mu^{\beta}}(y)| \leq \varepsilon$ and thus $\|F_{\mu^{\alpha}} - F_{\mu^{\beta}}\|_{\infty} \leq \varepsilon$, as desired.

\end{proof}

We note that the transform in Proposition \ref{transform} is not continuous on the entire simplex $\simp_N$ since the CDF of the measure associated to $\alpha \in \simp_N\setminus\simp_N^*$ is discontinuous. Now let $\mathcal{M}$ be the space of Borel probability measures on $[0,1]$ with the total variation norm, $\|\mu\|_{TV} = \sup_E|\mu(E)|$. We next show that the transform $\alpha \mapsto \mu^{\alpha} : \simp_N^* \rightarrow \mathcal{M}$ is continuous.

\begin{thm}
Let $\alpha \in \simp^*_N$. Then $\beta \rightarrow \alpha$ in $\simp_N$ if and only if $\mu^{\beta} \rightarrow \mu^{\alpha}$ in the total variation norm.
\end{thm}
\begin{proof}
Suppose $\beta \rightarrow \alpha$ in $\simp_N$. The implication of convergence in the total variation norm follows by proving the result for open intervals and passing to the regularity of the measure; however, the latter details are somewhat technical, so we provide a self-contained proof, herein. Let $\varepsilon > 0$. By Proposition \ref{transform}, there exists a $\delta > 0$ such that $\|F_{\mu^{\beta}} - F_{\mu^{\alpha}}\|_{\infty} < \varepsilon/2$ whenever $\|\beta - \alpha\|_{\infty} < \delta$. Let $\mathcal{O}$ be an open subset of $[0,1]$, and suppose $\{I_n\}_{n=1}^{\infty}$ is the disjoint collection of open intervals whose union is $\mathcal{O}$. Regarding $\mu^{\beta}$ and $\mu^{\alpha}$ as Riemann-Stieltjes measures, given $\eta > 0$, there exists a partition $\mathcal{P} = \{x_j\}$ of $I_n$ such that, by the triangle inequality,
\begin{align*}
\left|\mu^{\beta}(I_n) - \mu^{\alpha}(I_n)\right| &\leq \left|\mu^{\beta}(I_n) - \sum_{\mathcal{P}}(F_{\mu^{\beta}}(x_{j+1}) - F_{\mu^{\beta}}(x_j))\Delta x_j\right| \\
&\hspace{1cm}+ \left|\sum_{\mathcal{P}}(F_{\mu^{\beta}}(x_{j+1}) - F_{\mu^{\alpha}}(x_{j+1}))\Delta x_j\right| \\
&\hspace{1cm}+ \left|\sum_{\mathcal{P}}(F_{\mu^{\alpha}}(x_j) - F_{\mu^{\beta}}(x_j))\Delta x_j\right| \\
&\hspace{1cm}+ \left|\sum_{\mathcal{P}}(F_{\mu^{\alpha}}(x_{j+1}) - F_{\mu^{\alpha}}(x_j))\Delta x_j - \mu^{\alpha}(I_n)\right| \\
&\leq \eta + \varepsilon \lambda(I_n) + \eta
\end{align*}
where $\lambda$ is Lebesgue measure. Since $\eta$ was arbitrary, we have $\left|\mu^{\beta}(I_n) - \mu^{\alpha}(I_n)\right| \leq \varepsilon\lambda(I_n)$ and, thus,
\[
\left|\mu^{\beta}\left(\mathcal{O}\right) - \mu^{\alpha}\left(\mathcal{O}\right)\right| \leq \sum_{n=1}^{\infty}\left|\mu^{\beta}(I_n) - \mu^{\alpha}(I_n)\right| \leq \varepsilon\lambda(\mathcal{O}) \leq \varepsilon.
\]
Now let $E$ be a Borel-measurable subset of $[0,1]$, and let $\eta' > 0$. From the regularity of the measures, see \cite{Bogachev}, there exists an open set $\mathcal{O} \subset [0,1]$ containing $E$ such that
\[
\left|\mu^{\beta}(E) - \mu^{\alpha}(E)\right| \leq \left|\mu^{\beta}(E) - \mu^{\beta}(\mathcal{O})\right| + \left|\mu^{\beta}(\mathcal{O}) - \mu^{\alpha}(\mathcal{O})\right| + \left|\mu^{\alpha}(\mathcal{O}) - \mu^{\alpha}(E)\right| \leq \eta' + \varepsilon + \eta'.
\]
Since $\eta'$ was arbitrary, we have $\left|\mu^{\beta}(E) - \mu^{\alpha}(E)\right| \leq \varepsilon$. This concludes that $\mu^{\beta} \rightarrow \mu^{\alpha}$ in the total variation norm. \\
\\
Conversely, suppose that $\mu^{\beta} \rightarrow \mu^{\alpha}$ in the total variation norm. Then, by Proposition \ref{CDF}, we have
\[
\beta_k = \mu^{\beta}\left[\dfrac{k}{N},\dfrac{k+1}{N}\right] \rightarrow \mu^{\alpha}\left[\dfrac{k}{N},\dfrac{k+1}{N}\right] = \alpha_k,
\]
from which it immediately follows that $\beta \rightarrow \alpha$ in $\simp_N$. \\
\end{proof}

We conclude this section with a discussion of symmetric weighted Cantor measures.  As motivation, note that both CDF's in Figure \ref{CDFs} exhibit rotational symmetry about the point $(1/2,1/2)$.  First, we need a few definitions.  A Borel measure $\mu$ supported in the unit interval $[0,1]$ is said to be {\em symmetric} if $\mu(E) = \mu(1 - E)$ for every Borel-measurable set $E$. Here, if $E$ is Borel-measurable, then $1 - E := \{1 - x\,|\,x \in E\}$ is Borel-measureable since the collection of sets
\[
\left\{E\,\middle|\,1 - E \text{ is Borel-measurable}\right\}
\]
is a $\sigma$-algebra containing the open intervals.  We say that a weight vector $\alpha \in \simp_N$ is {\em palindromic} if $\alpha_{N-1-n} = \alpha_n$ for all $n \in \{0,1,...,N-1\}$.

\begin{thm}
Let $\alpha \in \simp_N$. The measure $\mu^{\alpha}$ is symmetric if and only if $\alpha$ is palindromic.
\end{thm}
\begin{proof}
If $\alpha_n = 1$ for some $n$, then $\mu^{\alpha}$ is a Dirac measure centered at $n/(N - 1)$. As such, the measure is symmetric only when $N = 2n + 1$, when $\alpha$ is palindromic. \\
\\
So we assume otherwise, that is, $\alpha_n < 1$ for all $n$. Suppose $\alpha$ is palindromic.  For any positive integer $k$ and $\vec{n} = (n_0,n_1,...,n_{k-1}) \in \{0,1,...,N-1\}^k$, let $I_{\vec n}$ be the open interval 
\[
I_{\vec{n}} := \left(\dfrac{1}{N^k}\sum_{\ell = 0}^{k-1}n_{\ell}N^{\ell}, \dfrac{1}{N^k}\left[1 + \sum_{\ell = 0}^{k-1}n_{\ell}N^{\ell}\right]\right).
\]
By Proposition \ref{CDF}, we have
\begin{align*}
\mu^{\alpha}\left(1 - I_{\vec n}\right) &= F_{\mu^{\alpha}}\left(1 - \dfrac{1}{N^k}\sum_{\ell=0}^{k-1}n_\ell N^\ell\right) - F_{\mu^{\alpha}}\left(1-\dfrac{1}{N^k}\left[1 + \sum_{\ell=0}^{k-1}n_\ell N^\ell\right]\right) \\
&= F_{\mu^{\alpha}}\left(\dfrac{1}{N^k}\left[1 + \sum_{\ell=0}^{k-1}(N - 1 - n_\ell)N^\ell\right]\right) - F_{\mu^{\alpha}}\left(\dfrac{1}{N^k}\sum_{\ell=0}^{k-1}(N - 1 - n_\ell)N^\ell\right) \\
&= \prod_{\ell=0}^{k-1} \alpha_{N - 1 - n_\ell} = \prod_{\ell=0}^{k-1} \alpha_{n_\ell} \\
&=  F_{\mu^{\alpha}}\left(\dfrac{1}{N^k}\left[1 + \sum_{\ell=0}^{k-1}n_\ell N^\ell\right]\right) - F_{\mu^{\alpha}}\left(\dfrac{1}{N^k}\sum_{\ell=0}^{k-1}n_\ell N^\ell\right) \\
&= \mu^{\alpha}(I_{\vec{n}}).
\end{align*}
As a consequence, any open set satisfies this identity by continuity of the measure. Then, from the regularity of $\mu^{\alpha}$, it follows that the measure is symmetric. \\
\\
Conversely, suppose $\mu^{\alpha}$ is symmetric. By Proposition \ref{CDF}, we have
\[
\alpha_{n_0} = \mu^{\alpha}\left(\dfrac{n_0}{N},\dfrac{n_0+1}{N}\right) = \mu^{\alpha}\left(\dfrac{N-1 - n_0}{N},\dfrac{N-1-n_0 + 1}{N}\right) = \alpha_{N - 1 - n_0}.
\]
Therefore, $\alpha$ is palindromic, completing the proof.

\end{proof}

%The last result of this section is a recursive formula for an orthogonal polynomial system attained via the Gram Schmidt algorithm. Even for general measures, we refer to the elements of the orthogonal system as Legendre polynomials. \\ 

%A edit: clarify what is meant by a Legendre polynomial here

The final observation of this section is a recursive formula for the monic Legendre polynomials associated to any symmetric, finite Borel measure $\mu$ on $[0,1]$, e.g. the ternary Cantor measure.  To be clear, we say that a sequence $(p_0,p_1,\dots)$ is a sequence of {\em Legendre polynomials} ({\em associated to $\mu$}) if each $p_k$ is a polynomial of degree $k$ so that for each $k\neq \ell$, $p_k$ and $p_\ell$ are orthogonal elements of $L_\mu^2[0,1]$.  Note that for each such measure $\mu$, this definition determines the family of Legendre polynomials uniquely up to scaling each polynomial.  
%We say that such a family $(p_0,p_1,\dots)$ is {\em monic} if each $p_k$ is monic, and {\em (ortho)normal} if each $p_k(x)$ has norm $1$ in $L_\mu^2[0,1]$.  %For convenience, we refer to the unique family of orthonormal Legendre polynomials with positive leading coefficients as {\em the} (orthonormal) Legendre polynomials (associated to $\mu$).
%\begin{fact}\label{leg fact}
%Let $\mu$ be any Borel probability measure $\mu$ on $[0,1]$.  Then any family of Legendre polynomials is a basis of $L^2[0,1]$.  Moreover the orthonormal family $(p_0(x), p_1(x),\dots)$ of Legendre polynomials is found by applying the Gram-Schmidt algorithm (in $L_\mu^2[0,1]$) to any sequence $(q_0(x), q_1(x),\dots)$ of polynomials with positive leading coefficients, and of degrees $0,1,2,\dots,$ respectively.  
%\end{fact}
\\
\\
Out of independent interest, we note the following $2$-term recursive formula for Legendre polynomials.

\begin{prop}\label{poly}
%Let $\mu$ be a symmetric Borel probability measure supported on $[0,1]$ and let $\{p_n\}_{n=0}^{\infty}$ be the orthonormal polynomials attained by the Gram Schmidt algorithm applied to $\{(x-1/2)^n\}_{n=0}^{\infty}$ in $L^2_{\mu}[0,1]$ with corresponding monic polynomials $\{m_n\}_{n=0}^\infty$. Then for all nonnegative integers $n$, 
%\[
%m_{n+2} = \left(x - \dfrac{1}{2}\right)m_{n+1} - \dfrac{\|m_{n+1}\|^2}{\|m_{n}\|^2}m_{n}.
%\]
Let $\mu$ be a symmetric, finite Borel measure on $[0,1]$ and let $(m_0,m_1,m_2,\dots)$ be the monic Legendre polynomials associated to $\mu$. Then $m_0(x) = 1$, $m_1(x) = x-1/2$, and $m_n$ alternates parity with respect to the line $x=1/2$.  Moreover, for all nonnegative integers $n$, 
\begin{equation}\label{symm recurrence}
m_{n+2}(x) = \left(x - \dfrac{1}{2}\right)m_{n+1}(x) - \dfrac{\|m_{n+1}\|_\mu^2}{\|m_{n}\|_\mu^2}m_{n}(x).
\end{equation}
\end{prop}

\begin{proof}

For convenience, we denote $q_1(x) := x - 1/2$. By way of the Gram Schmidt algorithm, we generate $m_{n + 1}(x)$ by subtracting off the projections of $q_1(x)m_n(x)$ on each of the monic Legendre polynomials up to degree $n$. For conciseness, we proceed by induction on $n\geq 0$, proving that (A) the parities of $m_n,$ $m_{n+1}$, and $m_{n+2}$ match the parities of $n, $ $n+1,$ and $n+2$, respectively, and that (B) Equation (\ref{symm recurrence}) holds.  %Since Equation (\ref{symm recurrence}) implies that $m_n$ alternates parity, it is sufficient to verify Equation (\ref{symm recurrence}) for each $n\geq 0$.  
\\
\\
We begin with the base case $n=0$.  Clearly $m_0(x) = 1$ and additionally, $m_0$ is even.  Since $\mu$ is symmetric and $q_1(x)m_0(x) = q_1(x)$ is odd, $\langle q_1m_0,m_0\rangle = 0$.  It follows that $m_1(x)$ is a (monic) constant multiple of $q_1(x)$, so $m_1(x) = x-1/2$ and $m_1$ is odd.  Finally, note that $\langle q_1m_1,m_1\rangle = 0$ since $q_1(x)m_1(x)m_1(x)$ is odd and $\mu$ is symmetric.  Since $m_2$ is monic, we need only subtract off the projection of $q_1(x)m_1(x)$ in the $m_0$ direction to find $m_2$.  So   
\begin{align*}
    m_2(x) &= q_1(x)m_1(x) - \dfrac{ \langle q_1m_1,m_0\rangle }{ \|m_0\|_{\mu}^2 }\, m_0(x) \\
    &= \left(x-\dfrac{1}{2}\right)m_1(x) - \dfrac{ \|m_1\|_{\mu}^2 }{ \|m_0\|_{\mu}^2 }m_0(x)
\end{align*}
and in particular, $m_2$ is even, so the base case of the claim holds.  \\
\\ 
Now suppose $n \geq 1$ and that the inductive hypothesis holds for $n-1$. Since $(m_0,m_1,m_2,\dots)$ is an orthogonal basis of $L_\mu^2[0,1]$ and $q_1(x)m_{n+1}(x)$ is a polynomial of degree $n+2$, it follows that we may write 
\[
q_1(x)m_{n+1}(x) = \sum_{k=0}^{n+2}c_km_k(x)
\]
for some constants $c_0,c_1,\dots,c_{n+2}$.  Note first that if $k\leq n-1$, then 
\[
\langle q_1m_{n+1}, m_k\rangle = \langle m_{n+1}, q_1m_k\rangle = 0
\]
since $m_{n+1}$ is orthogonal to any polynomial of degree less than $n+1$.  So $c_k = 0$ and we may write 
\[
q_1(x)m_{n+1}(x) = c_{n+2}m_{n+2}(x) + c_{n+1}m_{n+1}(x) + c_nm_n(x).  
\]
Since $q_1(x)m_{n+1}(x)$ and $m_{n+2}(x)$ are monic polynomials of degree $n+2$ and both $m_{n+1}$ and $m_n$ have lower degree, it follows that $c_{n+2} = 1$.  Finally, since $q_1(x)m_{n+1}(x)$ and $m_{n+1}$ have opposite parity, $\langle q_1m_{n+1}, m_{n+1}\rangle = 0$ and $c_{n+1} = 0$.  So finally 
\begin{align}
\left(x-\dfrac{1}{2}\right)m_{n+1}(x) = q_1(x)m_{n+1}(x) = m_{n+2}(x) + c_nm_n(x).  \label{eq recurrence}
\end{align}
By Equation (\ref{eq recurrence}) and the inductive hypothesis, it follows that the parity of $m_{n+2}$ matches the parity of $n+2$, so claim (A) holds.  For claim (B), it suffices to show that $c_n = \|m_{n+1}\|_{\mu}^2/\|m_n\|_{\mu}^2$.  By rearranging Equation (\ref{eq recurrence}) and considering the projection onto $m_n$, it follows that 
\[
c_n = \dfrac{ \langle q_1m_{n+1},m_n\rangle }{ \|m_n\|_{\mu}^2 } = \dfrac{ \langle m_{n+1},q_1m_n\rangle }{ \|m_n\|_{\mu}^2 }.  
\]
We conclude the calculation by first expanding $q_1(x)m_n(x)$, a polynomial of degree $n+1$, in terms of $m_0,m_1,\dots,m_{n+1}$. Thus $q_1(x)m_n(x) = \sum_{k=0}^{n+1}d_km_k(x)$ for some constants $d_0,d_1,...,d_{n+1}$. By inspecting the leading coefficient, it follows that $d_{n+1} = 1$ and by projecting onto the $m_{n+1}$ direction that $\langle q_1m_n,m_{n+1}\rangle = \|m_{n+1}\|_{\mu}^2$.  So $c_n = \|m_{n+1}\|_{\mu}^2 / \|m_n\|_{\mu}^2$, completing the induction and the proof.  

\end{proof}

Up to a translation factor, Proposition \ref{poly} is a reproduction of Bonnet's recurrence formula when the measure is Lebesgue. The drawback of Theorem \ref{poly} is that the algorithm is dependent on the norm of the monic polynomials.  One method to compute the norm of a polynomial is through the moments of the measure, which is the focus of Section \ref{sec moments}.  In Figure \ref{leg graphs}, we provide the graph of the first six normalized Legendre polynomials for the ternary Cantor measure.

\begin{figure}[ht]
\centering
    \includegraphics[width=1\textwidth]{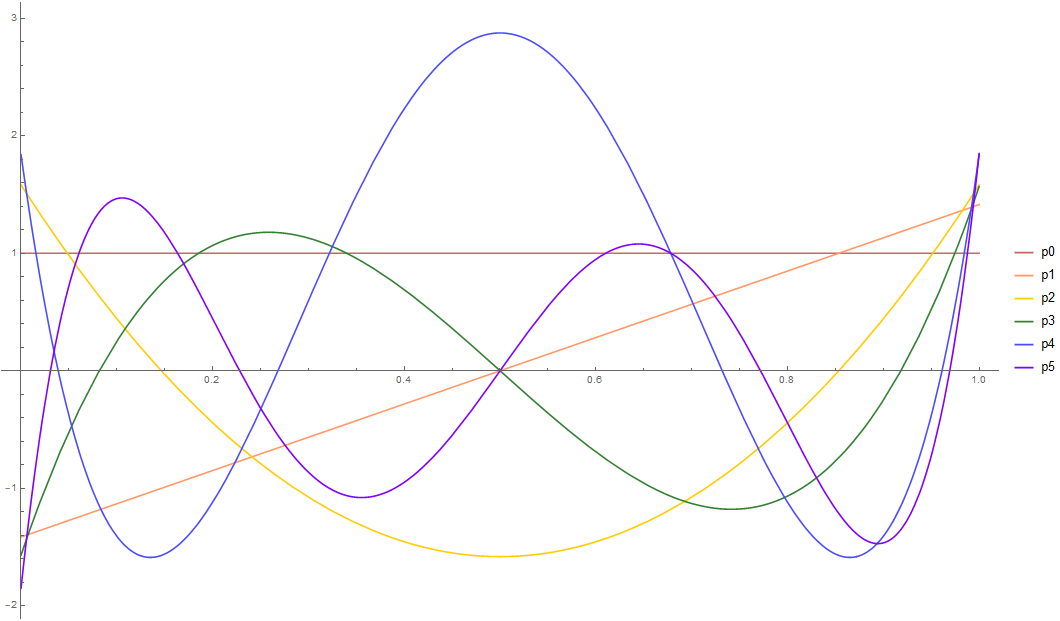}
    \caption{Selected normalized Legendre polynomials for the ternary Cantor measure}
    \label{leg graphs}
\end{figure}

\section{Moments of the weighted Cantor measure}\label{sec moments}
As previously observed, if $\alpha\in\simp_N$ is a standard basis vector, then $\mu^\alpha$ is a Dirac measure, and $L_{\mu^\alpha}^2[0,1]$ is $1$-dimensional.  Therefore, throughout this section, we focus mainly on $\alpha\in\simp_N^\ast$, but several results remain most general.  In this case, integration with respect to $\mu^\alpha$ presents a difficult calculation. One method is to interpret the problem as a Riemann-Stieltjes integral: for $f$ continuous on $[0,1]$,
\[
\int_0^1 f(x)\,d\mu^{\alpha}(x) = \int_0^1 f(x)\,dF_{\mu^{\alpha}}(x).
\]
Recall the sample set
\begin{align*}
S_k &= \left\{\dfrac{1}{N^k}\sum_{\ell=0}^{k-1}n_\ell N^\ell\,\middle|\,n_\ell \in \{0,1,...,N-1\}\right\}\cup\{1\} \\
&= \{0 = x_0 < x_1 < ... < x_{N^k - 1} < x_{N^k} = 1\}.
\end{align*}
By considering a uniform mesh size of $1/N^k$, we obtain the left-endpoint approximation of the above Riemann-Stieltjes integral,
\begin{align}\label{approx}
\sum_{j=0}^{N^k - 1} [F_{\mu^{\alpha}}(x_{j+1}) - F_{\mu^{\alpha}}(x_j)]f(x_j) = \sum_{n_0,n_1,...,n_{k-1}=0}^{N-1}\left(\prod_{\ell=0}^{k-1}\alpha_{n_\ell}\right)f\left(\dfrac{1}{N^k}\sum_{j=0}^{k-1}n_jN^j\right).
\end{align}
%The importance of this quantity is that it appears in one way or another in many results pertaining to moment calculation. 
For any $\alpha\in\simp_N$ and any nonnegative integer $m$, we define the {\em $m$-th moment of $\mu^\alpha$} to be 
\[
I_m^\alpha := \int_0^1x^m\,d\mu^{\alpha}(x).
\]
When the weight vector $\alpha$ is understood, we suppress the superscript on the moment notation. \\
\\
In the following proposition, we derive an invariance identity analogous to the invariance relation in Equation (\ref{invariance}). This identity will be essential for the remainder of the paper.

\begin{prop}
Let $f:[0,1] \rightarrow \mathbb{R}$ be integrable with respect to $\mu^{\alpha}$. Then
\begin{align}\label{identity}
\int_0^1 f(x)\,d\mu^{\alpha}(x) = \sum_{n = 0}^{N-1}\alpha_n\int_0^1(f\circ\varphi_n)(x)\,d\mu^{\alpha}(x)
\end{align}
where $\{\varphi_n\}_{n=0}^{N-1}$ is the associated IFS given by $\varphi_n(x) = (x + n)/N$.
\end{prop}

\begin{proof}
Since $\{\varphi_n\}_{n = 0}^{N - 1}$ are affine transformations, we note that the right-hand side of Equation ($\ref{identity})$ is well-defined. The proof follows by a standard bootstrapping argument. First observe, by Equation (\ref{invariance}), that (\ref{identity}) holds for any characteristic function,
\begin{align*}
\int_0^1\chi_E(x)\,d\mu^{\alpha}(x) &= \mu^{\alpha}(E) \\
&= \sum_{n = 0}^{N - 1} \alpha_n\mu^{\alpha}\left(\varphi^{-1}_n(E)\right) \\
&= \sum_{n = 0}^{N - 1}\alpha_n \int_0^1\chi_{\varphi^{-1}_n(E)}(x)\,d\mu^{\alpha}(x) \\
&= \sum_{n = 0}^{N - 1}\alpha_n\int_0^1(\chi_E \circ \varphi_n)(x)\,d\mu^{\alpha}(x).
\end{align*}
Then, by linearity of the integral, Equation (\ref{identity}) holds for simple functions. We attain the identity for nonnegative measurable functions by an application of the Simple Approximation and the Monotone Convergence Theorems; hence, the result follows in general by linearity of the integral.

\end{proof}

We now derive a recurrence relation for the moments of the weighted Cantor measure. We note that the relation exhibits the approximation in (\ref{approx}). While the relation was shown in \cite{JorgKornShu}, the proof of Theorem \ref{moments} as presented in this paper is original.

\begin{thm}\label{moments}
Let $\alpha \in \simp_N$, and let $k$ be a positive integer. Then $I_0 = 1$ and, for all $m \geq 1$,
\begin{equation}\label{recur}
I_m = \sum_{i=0}^{m-1}\binom{m}{i}\dfrac{N^{k(m-i)}}{N^{km}-1}I_i\sum_{n_0,n_1,...,n_{k-1}=0}^{N-1}\left(\prod_{\ell=0}^{k-1}\alpha_{n_\ell}\right)\left(\dfrac{1}{N^k}\sum_{j=0}^{k-1}n_jN^j\right)^{m-i}.
\end{equation}
In particular, 
\begin{equation}\label{moms}
I_m = \dfrac{1}{N^m-1}\sum_{n=0}^{N-1}\alpha_n\sum_{i=0}^{m-1}\binom{m}{i}n^{m-i}I_i.
\end{equation}
\end{thm}

\begin{proof}
Let $\beta = \alpha^{\otimes k}$ as in Proposition \ref{equivmeas}. Recall that we showed that $\mu^{\alpha} = \mu^{\beta}$ where the corresponding IFS for the weighted Cantor measure with respect to $\beta$ is $\{\psi_n\}_{n=0}^{N^k-1}$ given by $\psi_n(x) = (x + n)/N^k$ and $\beta_n = \prod_{\ell = 0}^{k - 1}\alpha_{n_{\ell}}$ where $n = n_0 + n_1N + ... + n_{k - 1}N^{k-1}$. Then, applying Equation (\ref{identity}) with respect to $\mu^{\beta}$, we have
\begin{align*}
\int_0^1 x^m\,d\mu^{\alpha}(x) &= \int_0^1 x^m\,d\mu^{\beta}(x) \\
&= \sum_{n = 0}^{N^k - 1}\beta_n\int_0^1 \left(\dfrac{x + n}{N^k}\right)^m\,d\mu^{\beta}(x) \\
&= \sum_{n_0,n_1,...,n_{k-1}=0}^{N-1}\left(\prod_{\ell=0}^{k-1}\alpha_{n_\ell}\right)\int_0^1 \left(\dfrac{1}{N^k}\left[x + \sum_{j=0}^{k-1}n_jN^j\right]\right)^m\,d\mu^{\alpha}(x).
\end{align*}
Next, we expand the product in the above integrand and rearrange the terms and sums.
\begin{align}
I_m &= \int_0^1 x^m\,d\mu^{\alpha}(x) \nonumber \\
&= \sum_{n_0,n_1,...,n_{k-1}=0}^{N-1}\left(\prod_{\ell=0}^{k-1}\alpha_{n_\ell}\right)\int_0^1 \sum_{i=0}^m \binom{m}{i} \dfrac{x^i}{N^{ki}}\left(\dfrac{1}{N^k}\sum_{j=0}^{k-1}n_jN^j\right)^{m - i}\,d\mu^{\alpha}(x) \nonumber \\
&= \sum_{i=0}^m\binom{m}{i}\dfrac{1}{N^{ki}}\int_0^1x^i\,d\mu^{\alpha}(x)\sum_{n_0,n_1,...,n_{k-1}=0}^{N-1}\left(\prod_{\ell=0}^{k-1}\alpha_{n_\ell}\right)\left(\dfrac{1}{N^k}\sum_{j=0}^{k-1}n_jN^j\right)^{m-i} \nonumber \\
&= \sum_{i=0}^m\binom{m}{i}\dfrac{1}{N^{ki}}I_i\sum_{n_0,n_1,...,n_{k-1}=0}^{N-1}\left(\prod_{\ell=0}^{k-1}\alpha_{n_\ell}\right)\left(\dfrac{1}{N^k}\sum_{j=0}^{k-1}n_jN^j\right)^{m-i}. \label{RiemannStieltjes}
\end{align}
When $i = m$ in (\ref{RiemannStieltjes}) the summand is $I_m/N^{km}$. We subtract this term from the left-hand side of the equation and solve for $I_m$ to attain the desired recurrence relation.

\end{proof}

For a reference to a large number of the moments $I_0, I_1, I_2,\dots$ of the ternary Cantor measure, see the Online Encyclopedia of Integer Sequence \cite{OEIS}.  Since each $I_k$ is rational, the sequence of numerators and denominators appear separately under A308612 and A308613, respectively.  Additionally, moments of the shifted ternary Cantor measure appear under A308614 and A308615. \\ 
\\
Instead of computing the moments recursively, we can individually approximate them from (\ref{approx}). The next result estimates the error of this approximation.

\begin{cor}
Let $\alpha \in \simp_N$ and let $\varepsilon > 0$. Fix an integer $m \geq 1$. If $k \geq \log_N\left(\dfrac{m}{\log(\varepsilon+1)}\right)$, then
\[
0 \leq I_m - \sum_{n_0,n_1,...,n_{k-1}=0}^{N-1}\left(\prod_{\ell=0}^{k-1}\alpha_{n_\ell}\right)\left(\dfrac{1}{N^k}\sum_{j=0}^{k-1}n_jN^j\right)^m < \varepsilon.
\]
\end{cor}

\begin{proof}
The first inequality follows from the observation that (\ref{approx}) is a lower approximation of the Riemann-Stieltjes integral $I_m$. For the upper bound, we manipulate ($\ref{RiemannStieltjes}$). Specifically, we subtract the term corresponding to $i = 0$ to obtain
\begin{align*}
I_m &- \sum_{n_0,n_1,...,n_{k-1}=0}^{N-1}\left(\prod_{\ell=0}^{k-1}\alpha_{n_\ell}\right)\left(\dfrac{1}{N^k}\sum_{j=0}^{k-1}n_jN^j\right)^m \\
&= \sum_{i=1}^{m}\binom{m}{i}\dfrac{1}{N^{ki}}I_i\sum_{n_0,n_1,...,n_{k-1}=0}^{N-1}\left(\prod_{\ell=0}^{k-1}\alpha_{n_\ell}\right)\left(\dfrac{1}{N^k}\sum_{j=0}^{k-1}n_jN^j\right)^{m-i} \\
&\leq \sum_{i=1}^{m}\binom{m}{i}\dfrac{1}{N^{ki}}I_iI_{m-i}.
\end{align*}
Using $I_i < 1$, we have
\[
\sum_{i=1}^{m}\binom{m}{i}\dfrac{1}{N^{ki}}I_iI_{m-i} < \sum_{i=1}^{m}\binom{m}{i}\dfrac{1}{N^{ki}} = \left(1 + \dfrac{1}{N^k}\right)^m - 1 \leq \exp\left(\dfrac{m}{N^k}\right) - 1 \leq \varepsilon,
\]
as desired.

\end{proof}

We define the \textit{Laplace transform} of a finite measure $\mu$ on $[0,1]$ as the function on $\mathbb{R}$ given by
\[
\mathcal{L}_{\mu}(s) = \int_0^1e^{-sx}\,d\mu(x).
\]
Here, we use the Laplace transform of a weighted Cantor measure to approach the moment problem.

\begin{thm}\label{entire}
Let $\alpha \in \simp_N$. The infinite product
\[
f(z) := \prod_{r=1}^{\infty}\sum_{n=0}^{N-1}\alpha_n\exp\left(-\dfrac{nz}{N^r}\right)
\]
is well-defined for $z \in \mathbb{C}$. Furthermore, $f$ is entire and $\mathcal{L}_{\mu^{\alpha}}(s) = f(s)$ for $s \in \mathbb{R}$.
\end{thm}

\begin{proof}
Using the triangle inequality, the power series for $\exp(\cdot)$, and Tonelli's theorem, we have
\begin{align*}
\sum_{r = 1}^{\infty}\left|\left[\sum_{n=0}^{N-1}\alpha_n\exp\left(-\dfrac{nz}{N^r}\right)\right] - 1\right| &= \sum_{r = 1}^{\infty}\left|\sum_{n=0}^{N-1}\alpha_n\left[\exp\left(-\dfrac{nz}{N^r}\right) - 1\right]\right| \\
&\leq \sum_{r = 1}^{\infty}\sum_{n=0}^{N-1}\alpha_n\left|\exp\left(-\dfrac{nz}{N^r}\right) - 1\right| \\
&= \sum_{r = 1}^{\infty}\sum_{n=0}^{N-1}\alpha_n\left|\sum_{k=1}^{\infty}\dfrac{(-1)^kn^kz^k}{N^{rk}k!}\right| \\
&\leq \sum_{r = 1}^{\infty}\sum_{n=0}^{N-1}\alpha_n\sum_{k=1}^{\infty}\dfrac{n^k|z|^k}{N^{rk}k!} \\
&= \sum_{n=0}^{N-1}\alpha_n\sum_{k = 1}^{\infty}\sum_{r=1}^{\infty}\dfrac{n^k|z|^k}{N^{rk}k!} \\
&= \sum_{n=0}^{N-1}\alpha_n\sum_{k = 1}^{\infty}\dfrac{n^k|z|^k}{(N^k - 1)k!}.
\end{align*}
We observe that the last sum converges by the ratio test. From this, it follows that $f$ is well-defined and entire. \\
\\
Applying Equation (\ref{identity}) to $f(x) = e^{-sx}$, we find
\[
\mathcal{L}_{\mu^{\alpha}}(s) = \sum_{n=0}^{N-1}\alpha_n\int_0^1\exp\left(-s\left[\dfrac{x+n}{N}\right]\right)\,d\mu^{\alpha}(x) = \mathcal{L}_{\mu^{\alpha}}\left(\dfrac{s}{N}\right)\sum_{n=0}^{N-1}\alpha_n\exp\left(-\dfrac{ns}{N}\right).
\]
Then, from an argument by induction, we have
\[
\mathcal{L}_{\mu^{\alpha}}(s) = \mathcal{L}_{\mu^{\alpha}}\left(\dfrac{s}{N^k}\right)\prod_{r=1}^k\sum_{n=0}^{N-1}\alpha_n\exp\left(-\dfrac{ns}{N^r}\right).
\]
From the Bounded Convergence Theorem, we have $\lim_{k\to\infty}\mathcal{L}_{\mu^{\alpha}}\left(\dfrac{s}{N^k}\right) = 1$, and the desired identity follows.

\end{proof}

The \textit{moment generating function} (MGF) $G_\alpha(s)$ is defined analogously, 
\begin{align}\label{gen}
G_{\alpha}(s) := \mathcal{L}_{\mu^\alpha}(-s).
\end{align}
It can be seen that
\[
G_{\alpha}(s) = \sum_{ m=0 }^\infty I_m\frac{s^m}{m!}.
\]
We may derive many interesting identities from $G_\alpha(s)$, such as the following recurrence relation.

\begin{prop}\label{palin mom}
Let $\alpha \in \simp_N$ be palindromic, and let $m$ be an odd integer. Then
\[
I_m = \dfrac{1}{2}\sum_{k=0}^{m-1}(-1)^k\binom{m}{k}I_k.
\]
\end{prop}
\begin{proof}
From Theorem \ref{entire} and the assumption that $\alpha$ is palindromic, we find
\begin{align*}
G_{\alpha}(s) &= \prod_{r=1}^{\infty}\sum_{n=0}^{N-1}\alpha_{N-1-n}\exp\left(\dfrac{(N-1-n)s}{N^r}\right) \\
&= \prod_{r=1}^{\infty}\exp\left(\dfrac{(N-1)s}{N^r}\right)\sum_{n=0}^{N-1}\alpha_n\exp\left(-\dfrac{ns}{N^r}\right) \\
&= e^sG_{\alpha}(-s).
\end{align*}
This identity, in terms of the power series expansion of $G_{\alpha}(s)$ and of $e^s$, is then
\[
\sum_{m=0}^{\infty}\dfrac{I_m}{m!}s^m = \left(\sum_{m=0}^{\infty}\dfrac{1}{m!}s^m\right)\left(\sum_{m=0}^{\infty}(-1)^m\dfrac{I_m}{m!}s^m\right) = \sum_{m=0}^{\infty}\left(\sum_{k=0}^m\dfrac{(-1)^kI_k}{k!(m-k)!}\right)s^m.
\]
From the uniquess of the coefficients, we have
\[
I_m = \sum_{k=0}^m(-1)^k\binom{m}{k}I_k
\]
from which the desired identity follows.

\end{proof}

Viewing $G_{\alpha}$ as a function on $\mathbb{C}$, i.e. $G_{\alpha}(z) = f(-z)$ for $f$ in Theorem \ref{entire}, we note that $G_\alpha$ is entire. A useful consequence of this viewpoint is in estimating the moments. Specifically, we consider the partial product approximations defined for all nonnegative integers $k$,
\[
G_{\alpha;k}(z) := \prod_{r = 1}^k\sum_{n=0}^{N-1}\alpha_n\exp\left(\dfrac{nz}{N^r}\right) = \sum_{m=0}^\infty I_{m;k}\frac{z^m}{m!}.
\]
For any nonnegative integer $m$, we note that $0\leq I_{m;k}\nearrow I_m$.  Indeed, this follows immmediately from the fact that $G_{\alpha;0}(z) = 1$ and, for all $k\geq 0$, $G_{\alpha;k+1}(z)$ is the product of $G_{\alpha;k}(z)$ and a power series centered at $z=0$ with nonnegative coefficients and constant term $1$.  

% make two observations. 
%Expanding the product, we observe that $I_{m;k}$ are nonnegative. 

%To see that $I_{m;k}\nearrow I_k$, we Next we differentiate $G_{\alpha;k+1}$ and then evaluate at $z = 0$ to obtain
%\[
%I_{m;k+1} = \sum_{l = 0}^m\dfrac{1}{N^{(k+1)(m - l)}}\binom{m}{l}I_{l;k}\sum_{n = 0}^{N - 1}\alpha_nn^{m-l} \geq I_{m;k}.
%\]
%Thus, we observe that $I_{m;k}\nearrow I_m$ as $k %\rightarrow \infty$.

\begin{thm}\label{cauchy est}
Let $\alpha \in \simp_N$. For any positive integers $m,k$ with $m \geq 2$, 
\[
\left|I_{m}-I_{m;k}\right| \leq \dfrac{em\sqrt{m-1}}{N^k}.  
\]
\end{thm}

\begin{proof}
We first verify the following as an identity of formal power series, for all positive integers $k$.  
\begin{equation}\label{gk identity}
G_{\alpha}(z)% - G_{\alpha;k}(z) 
= G_{\alpha;k}(z)\cdot 
%\left(
G_{\alpha}\left(\dfrac{z}{N^{k}}\right)% - 1\right)
\end{equation}
Indeed, 
\begin{align*}
    G_{\alpha;k}(z) \cdot G_{\alpha}\left( \dfrac{z}{N^k} \right) = \left( \prod_{r=1}^k\sum_{n=0}^{N-1}\alpha_n\exp\left(  \dfrac{nz}{N^r}\right) \right)\cdot \left( 
    \prod_{r=k+1}^\infty\sum_{n=0}^{N-1}\alpha_n\exp\left(  \dfrac{nz}{N^r}\right)
    \right) = \prod_{r=1}^\infty\sum_{n=0}^{N-1}\alpha_n\exp\left(  \dfrac{nz}{N^r}\right)
\end{align*}
so by definition of $G_\alpha(z)$, Equation (\ref{gk identity}) holds as formal power series.  To see that Equation (\ref{gk identity}) holds analytically, it is sufficient to note that both $G_{\alpha;k}(z)$ and $G_\alpha(z)$ are entire functions.  \\
\\
Now let $m,k$ be positive integers with $m\geq 2$ and let $R > 0$.  By subtracting $G_{\alpha;k}(z)$ from both sides of Equation (\ref{gk identity}), we obtain the analytic identity 
\begin{equation}\label{gk difference}
G_{\alpha}(z) - G_{\alpha;k}(z) = G_{\alpha;k}(z)\left( G_{\alpha}\left( \dfrac{z}{N^k} \right) -1 \right).  
\end{equation}
Since the coefficients of $z^m$ in $G_\alpha(z)$ and $G_{\alpha;k}(z)$ are $I_m/m!$ and $I_{m;k}/m!$, respectively, it follows from the Cauchy integral formula that 
\begin{align*}
    \left|I_{m}-I_{m;k}\right| = \frac{m!}{2\pi} \left| \, \int_{ |z| = R }
    \frac{G_\alpha(z) - G_{\alpha;k}(z)}{z^{m+1}}
    \,dz
    \right| \leq \frac{m!}{R^m}\,\max_{|z|=R}
    \left| 
    G_{\alpha;k}(z)\right|
    \cdot 
    \max_{|z|=R}
    \left|
    G_\alpha\left(\dfrac{z}{N^{k}}\right) - 1
    \right|.
\end{align*}
Note that $G_{\alpha;k}(z)$ is a power series with nonnegative coefficients, so it follows that the first maximum is attained by setting $z=R$.  Likewise since $G_\alpha(z/N^k)(z)-1$ is a power series with nonnegative coefficients, it follows also that the second maximum is attained by setting $z = R$. Continuing the calculation, 
\begin{align*}
    |I_m - I_{m;k}| &\leq \frac{m!}{R^m}\cdot G_{\alpha;k}(R)\left( G_\alpha\left(\dfrac{R}{N^{k}}\right) - 1 \right)\\
    &\leq \frac{m!}{R^m}\cdot\exp\left(R\left(1-\dfrac{1}{N^k}\right) \right)
    \left(
    \exp\left( 
    \dfrac{R}{N^{k}}
    \right)-1
    \right) \\
    &= \dfrac{m!}{R^m}\cdot e^R \left(1 - \exp\left(-\dfrac{R}{N^k}\right)\right) \\
    &\leq \dfrac{m!}{N^k}  \cdot \dfrac{e^R}{R^{m-1}} \\
    &= \dfrac{m!}{N^k}\cdot f(R)
\end{align*}
where $f:(0,\infty)\to\mathbb{R}$ is defined by $f(x):= e^x x^{1-m}$ and the second inequality follows by noting that $G_\alpha(x) \leq G_{(0,0,...,1)}(x) = e^x$ for $x > 0$ and similarly for $G_{\alpha;k}$.  We minimize this upper bound (for $m,k$ fixed) using elementary calculus.  Note first that $f(x)$ is differentiable on $(0,\infty)$ and $f(x)\to \infty$ as $x$ tends to either endpoint.  Since 
\[
f'(x) = e^x x^{-m}\left( (1-m) + x \right), 
\]
it follows that $f$ is minimized at $x = m-1$.  Evaluating $f(m-1)$, we have
\[
|I_m - I_{m;k}| \leq \dfrac{m!}{N^k}\cdot \left(\dfrac{e}{m-1}\right)^{m-1}.
\]
Applying the Stirling approximation $(m-1)!\leq e\sqrt{m-1}\cdot \left( \dfrac{m-1}{e} \right)^{m-1}$, the desired identity holds.  

\end{proof}

\begin{rem}\label{shifted MGF}
For any palindromic weight vector $\alpha$, it is suitable to alternatively define the moment generating function under $\nu^{\alpha}$, the measure defined by shifting $\mu^\alpha$ from $[0,1]$ to $[-1/2,1/2]$.  Then the moment generating function with respect to $\nu^{\alpha}$ satisfies
\begin{align*}
\sum_{m=0}^\infty J_m\frac{s^m}{m!} = H_\alpha(s):= \int_{ -1/2}^{1/2} e^{sx} \,d\nu^{\alpha}(x) = \int_0^1 e^{s(x - 1/2)}\,d\mu^\alpha(x) =   e^{-s/2}G_\alpha(s).  
\end{align*}
With some careful manipulation, we may then analogously define $H_{\alpha;k}(s) = \sum_{m=0}^\infty J_{m;k}\frac{s^m}{m!}$ as the partial product $\prod_{r=1}^k e^{ \frac{-(N-1)s}{2N^r} } \sum_{n = 0}^{N-1}\alpha_n \exp\left( \frac{n s}{N^r} \right)$.  Distributing, each product is now a weighted average of hyperbolic cosines of the form $\cosh ( \frac{\ell s}{N^r} )$, where each $\ell$ is a half integer between $0$ and $N/2$.
\end{rem}

\begin{rem}\label{algorithm}
For any $m\geq 0$, we may estimate the coefficients $I_1,\dots,I_m$ $($or $J_1,\dots,J_m$ for a palindromic weight vector$)$ within uniform error at most $\varepsilon> 0$ in $O(\log\log(1/\varepsilon)\cdot m\log m)$.  This is a substantial improvement when compared to the exact computation of each $I_1,\dots,I_m$ $($or $J_1,\dots,J_m)$ from Proposition \ref{palin mom}, which runs in $O(n^2)$.  \\
\\We describe the details for moments under $\mu^\alpha$.  First, apply Theorem \ref{cauchy est} to select $k = k(\varepsilon) = O(\log(1/\varepsilon))$ so that $|I_m-I_{m;k}|\leq \varepsilon$.  Writing $f(s) = \sum_{n = 0}^m c_ns^n$ for the truncation of $\sum_{ \ell = 0 }^{N-1}\alpha_\ell\exp\left( \frac{\ell s}{N}\right)$ to degree $m$, it follows that $G_{\alpha;k}(s)$ and $F(s) := f(s)f(s/N)\cdots f(s/N^{k-1})$ have identical coefficients up to degree $m$.  For algorithmic simplicity, we may assume that $k$ is a power of two, but this assumption may be circumvented with some care, or absorbed as a factor in $O(\log(1/\varepsilon))$. We provide the following pseudocode. \\
\begin{enumerate}%[(a)]
    \item $F(s) \leftarrow \sum_{n=0}^m \sum_{ \ell =0 }^{N-1} \alpha_\ell\cdot (s\ell/N)^n / n! $. 
    \item $j\rightarrow 0$. 
    \item \label{alg goto} If $2^j = k$, go to step (\ref{alg end}). 
    \item $F(s)\rightarrow F(s)\cdot F(s/N^{2^j})$. 
    \item Truncate $F(s)$ to degree $m$ in $s$. 
    \item $j\rightarrow j+1$.  
    \item Go to step (\ref{alg goto}).  
    \item \label{alg end} Return $F(s)$
\end{enumerate}
Since a successful termination performs $\log_2(k)$ products of degree $m$ polynomials, by using a Fast Fourier Transform, the overall complexity is reduced to $O(\log\log(1/\varepsilon)\cdot m\log m)$, as desired.
\end{rem}

In \cite{Grabner}, Grabner and Prodinger investigated measures whose distributions are given by Cantor sets and are somewhat similar to the ternary Cantor measure yet in general do not arise from an IFS.  The major result in their paper is the following asymptotic behavior of the corresponding moments,
\[
I_m = F(\log_{1/\theta}m)m^{-\log_{1/\theta}(2)}\left(1 + O\left(\dfrac{1}{m}\right)\right)
\]
where $F(x)$ is a periodic function of period 1 and known Fourier coefficients. In regards to this paper, the ternary Cantor measure is ascertained by letting $\theta = 1/3$. The final result of this paper is a lower bound approximation for the rate of decay of the moments for a weighted Cantor measure. It is intriguing that the bound that we obtain is precisely of the same order as the result of Grabner and Prodinger.

\begin{thm}
Let $\alpha \in \simp_N^*$.  If $\alpha_{N-1} = 0$, then $I_m\leq \left( \dfrac{N-1}{N} \right)^m$ for $m \geq 0$.  Otherwise, there exists a constant $C(\alpha) > 0$ such that $I_m\geq C(\alpha)m^{-\gamma}$ for all $m \geq 1$ where $\gamma = \log_N(1/\alpha_{N-1})$.
\end{thm}

\begin{proof}
Suppose $\alpha_{N-1} = 0$. From the invariance relation in Equation (\ref{invariance}), we observe that the support of $\mu^{\alpha}$ is contained in $[0,1-1/N]$. Then, for any nonnegative integer $m$,
\[
I_m = \int_0^1 x^m\,d\mu^{\alpha}(x) = \int_0^{1 - 1/N}x^m\,d\mu^{\alpha}(x) \leq \left(\dfrac{N-1}{N}\right)^m.  
\]
So the first claim holds.  Now suppose $\alpha_{N-1} > 0$. We assume without loss of generality that $m > \gamma(N-1)$ to establish $C(\alpha) > 0$ which may be adjusted to compensate for the remaining (finitely many) moments. Now note for all positive integers $k$, $\mu^\alpha[1-1/N^{k},1] = (\alpha_{N-1})^k$, so 
\[
I_m = \int_0^1x^m\,d\mu^{\alpha}(x) \geq \int_{1-N^{-k}}^1x^m\,d\mu^{\alpha}(x) \geq (1 - N^{-k})^m(\alpha_{N-1})^k = (1-N^{-k})^m\cdot N^{-k\gamma} = f(k), 
\]
where $f:(0,\infty)\to (0,\infty)$ is defined by $f(x):= N^{-\gamma x} (1-N^{-x})^m$.  In order to maximize this lower bound on $I_m$, we appeal to elementary calculus to first optimize the differentiable function $f$ on $(0,\infty)$ and then select the most optimal positive integer $k$, for a given $m$. From logarithmic differentiation, we find
\begin{align*}
\dfrac{ f'(x) }{ f(x) } &= \left( -\gamma x\log N + m\log\left( 1-N^{-x} \right) \right)' \\
&= -\gamma \log N + m \dfrac{N^{-x}\log N}{1-N^{-x}} \\
&= -\gamma\log N + \dfrac{m\log N}{N^x-1} \\
&= \left( -\gamma + \dfrac{m}{N^x-1} \right)\log N, 
\end{align*}
so $f'$ has its unique zero at $x_0 = \log_N\left( 1+\dfrac{m}{\gamma} \right) > 0$.  In fact, by the assumption on $m$, 
\[
x_0 = \log_N\left( 1+\dfrac{m}{\gamma} \right) > \log_N\left( 1+\dfrac{\gamma(N-1)}{\gamma} \right) = \log_N(N) = 1.
\]
Note that $f'>0$ on $(0,x_0)$ and $f'<0$ on $(x_0,\infty)$, so that $f(x)$ is maximized over $(0,\infty)$ at $x_0$.  Moreover, by monotonicity of $f$ on $(0,x_0)$ and $(x_0,\infty)$ and the fact that $x_0 > 1$, it also follows that the optimal integer is either $\lceil x_0\rceil$ or $\lfloor x_0\rfloor$.  Write $k_0 = k_0(m)$ for the positive integer which maximizes $f$.  We now show that the ratio of $f(x_0)$ and $f(k_0)$ is bounded above and below by constants (depending only on $\alpha$).  Let $\varepsilon = \varepsilon(m) :=k_0-x_0 \in (-1,1)$.  Note that $\gamma(1-N^{-\varepsilon}) \in ( \gamma(1-N), \gamma(1-1/N) )$ and
\begin{align*}
\dfrac{f(k_0)}{f(x_0)} &= N^{-\gamma \varepsilon} \left( \dfrac{ 1-N^{-k_0} }{ 1-N^{-x_0} } \right)^m \\
&= N^{-\gamma\varepsilon } \left( 1+\dfrac{1-N^{-\varepsilon}}{N^{x_0}-1} \right)^m \\
&= N^{-\gamma\varepsilon } \left( 1+\dfrac{\gamma(1-N^{-\varepsilon})}{m} \right)^m \\
&\geq N^{-\gamma\varepsilon} C_1e^{\gamma (1-N^{-\varepsilon})}
\end{align*}
for some $C_1 > 0$ depending only on $N$ and $\gamma$, i.e. $\alpha$. The last inequality follows from observing that sequence of functions $\{(1 + x/m)^m\}$ are positive and converge uniformly to $e^x$ on $(\gamma(1-N), \gamma(1-1/N))$. Further note that
\[
f(x_0) = N^{-\gamma x_0} \left( 1-N^{-x_0} \right)^m 
= \left( 1+\dfrac{m}{\gamma} \right)^{-\gamma} \left( 1-\dfrac{\gamma}{\gamma+m} \right)^m
\geq \left( \dfrac{m}{\gamma}+1 \right)^{-\gamma}e^{ -\gamma }.
\]
Thus, we find the bound
\begin{align*}
I_m &\geq f(k_0) \\
&\geq N^{-\gamma\varepsilon}C_1e^{\gamma(1 - N^{-\varepsilon})}f(x_0) \\
&\geq N^{-\gamma\varepsilon}C_1e^{\gamma(1 - N^{-\varepsilon})}\left( \dfrac{m}{\gamma}+1 \right)^{-\gamma}e^{ -\gamma } \\
&= N^{-\gamma\varepsilon}C_1e^{-\gamma N^{-\varepsilon}}m^{-\gamma}\left( \dfrac{1}{\gamma}+\dfrac{1}{m} \right)^{-\gamma} \\
&\geq N^{-\gamma}C_1e^{-\gamma N}\left( \dfrac{1}{\gamma}+1\right)^{-\gamma}m^{-\gamma}.
\end{align*}

\end{proof}

\begin{rem}
Under the shifted measure $\nu^\alpha$ defined in Remark \ref{shifted MGF}, the moments decay exponentially {\em regardless} of weight vector $\alpha$.  Indeed,
\[
|J_m| = \left|\,\int_{-1/2}^{1/2}x^m\,d\nu^\alpha(x)\right| = \left|\,\int_0^1(x - 1/2)^m\,d\mu^\alpha(x)\right| \leq \left(\dfrac{1}{2}\right)^m.  
\]
\end{rem}

\section{Acknowledgements}
Steven N. Harding was supported in part by the National Science Foundation and the National Geospatial-Intelligence Agency under the NSF award \#1832054. \\
\\Alexander W.~N.~Riasanovsky was supported in part by the ISU Mathematics Department Lambert Research Fellowship.

\end{document}